\documentclass{amsart}

\usepackage{amsmath,latexsym,amssymb,amsthm,graphicx}
\usepackage{subfigure}
\usepackage{xfrac}
\usepackage{fix-cm}
\usepackage{hyperref,doi}
\usepackage{stmaryrd}
\usepackage{multirow}
\usepackage{tabularx}
\usepackage[all]{xy}
\usepackage{enumitem}

\usepackage{color}

\setenumerate[1]{label=(\thesection.\arabic*)}

\numberwithin{equation}{section}
\numberwithin{figure}{section}

\newtheorem{theorem}{Theorem}[section]
\newtheorem*{maintheorem}{Main Theorem}
\newtheorem{lemma}[theorem]{Lemma}

\newtheorem{corollary}[theorem]{Corollary}
\newtheorem{proposition}[theorem]{Proposition}

\theoremstyle{definition}
\newtheorem{remark}[theorem]{Remark}
\newtheorem{remarks}[theorem]{Remarks}
\newtheorem{example}[theorem]{Example}
\newtheorem{definition}[theorem]{Definition}

\def\Z{\ensuremath{\mathbb{Z}}}

\def\R{\ensuremath{\mathbb{R}}}

\newcommand{\pa}[1]{\left(#1\right)}
\newcommand{\cpa}[1]{\left\{#1\right\}}
\newcommand{\wt}[1]{\widetilde{#1}}
\newcommand{\tn}[1]{\textnormal{#1}}

\newcommand{\im}[1]{\tn{Im} \, #1}
\newcommand{\Int}[1]{\tn{Int} \, #1}

\newcommand{\la}[2]{\mathbb{L}\pa{#1,#2}}
\newcommand{\ci}[3]{H^{#1}_{\infty}\pa{#2;#3}}
\newcommand{\rci}[3]{\wt{H}^{#1}_{\infty}\pa{#2;#3}}
\newcommand{\rc}[3]{\wt{H}^{#1}\pa{#2;#3}}
\newcommand{\rh}[3]{\wt{H}_{#1}\pa{#2;#3}}

\newcommand{\ilim}{\mathop{\varinjlim}\limits}
\newcommand{\csi}{\mathbin{\natural}}
\newcommand{\cs}{\mathbin{\#}}
\newcommand{\csb}{\mathbin{\#_\partial}}

\renewcommand{\hom}[3]{\tn{Hom}_{#1}\pa{#2,#3}}

\font\cuf=cmtt8
\newcommand{\curl}[1]{{\cuf #1}}

\begin{document}
\title{Connected sum at infinity and $4$-manifolds}

\author[J.~Calcut]{Jack S. Calcut}
\address{Department of Mathematics\\
         Oberlin College\\
         Oberlin, OH 44074}
\email{jcalcut@oberlin.edu}
\urladdr{\href{http://www.oberlin.edu/faculty/jcalcut/}{\curl{http://www.oberlin.edu/faculty/jcalcut/}}}

\author[P.~Haggerty]{Patrick V. Haggerty}
\address{Department of Mathematics\\
         Oberlin College\\
         Oberlin, OH 44074}
\email{phaggert@oberlin.edu}

\date{\today}

\begin{abstract}
We study connected sum at infinity on smooth, open manifolds.  This operation requires a choice of proper ray in each manifold summand.  In favorable circumstances, the connected sum at infinity operation is independent of ray choices.  For each $m\geq 3$, we construct an infinite family of pairs of $m$-manifolds on which the connected sum at infinity operation yields distinct manifolds for certain ray choices.  We use cohomology algebras at infinity to distinguish these manifolds.
\end{abstract}

\maketitle

\section{Introduction}
\label{sec:introduction}

There exist several natural operations for combining manifolds.
These include classical connected sum (CS), classical connected sum boundary (CSB), and the less familiar connected sum at infinity (CSI).
CSI is roughly what happens to manifold interiors under CSB.

The CSI operation, also called end sum, was introduced by Gompf (1985) \cite{gompf} to study smooth manifolds homeomorphic to $\R^4$.
CSI is now a major tool for studying exotic smooth structures on open $4$-manifolds \cite[\S9.4]{gs}, \cite{gompf13}.
It was also used by Ancel (1980s, unpublished) to study Davis manifolds and by Tinsley and Wright (1997) and by Myers (1999)~\cite{myers} to study $3$-manifolds.
Recently, Calcut, King, and Siebenmann (2012)~\cite{cks} gave a general treatment of CSI that yielded a natural proof of the Cantrell-Stallings hyperplane unknotting theorem.

Each of the above operations involves some choices.
Under mild restrictions, CS and CSB are independent of the choices~\cite[\S2]{cks}.
CSI requires a choice of proper ray in each manifold summand.
As a ray knots in $\R^m$ if and only if $m=3$, it is not surprising that the result of CSI depends on ray choices in dimension 3.
In fact, one may construct such examples where $\R^3$ is summed with itself for various rays (see Myers~\cite{myers} and the appendix below).

For one-ended manifolds of dimension $m\geq 4$, the binary CSI operation yields a unique manifold up to diffeomorphism provided either:
\begin{enumerate}
\item\label{eq:cbs} One summand is smoothly collared at infinity by $S^{m-1}$.

\item\label{eq:ML} Both summands satisfy the Mittag-Leffler condition. The Mittag-Leffler condition holds on a manifold $M$, for instance, if: (i) $M$ is topologically collared at infinity or (ii) $M$ admits an exhausting Morse function with only finitely many coindex 1 critical points.

\end{enumerate}
Proofs of these two statements will appear in a subsequent paper.

The main purpose of this paper is to prove the following, 
which answers affirmatively a conjecture of Siebenmann \cite{cks}.

\begin{maintheorem}
    There exist infinitely many pairs $M$ and $N$ of open, one-ended 4-manifolds such that ray choice alters the proper homotopy type of the CSI of $M$ and $N$.
\end{maintheorem}

In our explicit examples, one CSI summand is collared at infinity and thus contains a unique ray up to ambient isotopy.
So, ray choice is relevant even in just one summand.
In view of \ref{eq:cbs} and \ref{eq:ML},
our examples are, in a sense, the simplest possible.

The question arises: given a cardinal number $c$, does there
exist an open, one-ended $4$-manifold $M$ such that the CSI of $M$ with itself yields at least $c$ manifolds up to proper homotopy?
In Section~\ref{sec:generalization}, we exhibit an infinite collection of manifolds answering this question in the affirmative for each at most countably infinite $c$.
We conjecture that this question has an affirmative answer when $c$ is uncountable.

In each example used to prove the Main Theorem, our ray choices do not alter the homotopy type of the CSI sum.
We conjecture that there exist open, one-ended manifolds $M$ and $N$ such that ray choice alters the homeomorphism type but not the proper homotopy type of the CSI of $M$ and $N$.  Further, we conjecture that there exist open, one-ended $4$-manifolds $M$ and $N$ such that ray choice alters the diffeomorphism type but not the homeomorphism type of the CSI of $M$ and $N$.  A possible candidate is the CSI of a ladder manifold (as defined in Section~\ref{sec:ladders}) and some exotic $\R^4$, although distinguishing the resulting manifolds up to diffeomorphism seems to be beyond present $4$-manifold technology.

This paper is organized as follows.
Section~\ref{sec:notation} defines CSI and fixes some notation.
Section~\ref{sec:ladders} introduces ladder manifolds and computes their cohomology algebras at infinity.
Section~\ref{sec:stringer_sum} defines stringer sum, an operation related to CSI.
Section~\ref{sec:lens} studies stringer sum for ladder manifolds based on lens spaces.
Section~\ref{sec:proof} proves the Main Theorem.
Section~\ref{sec:generalization} presents various ways of generalizing our examples, including a proof of the Main Theorem in each dimension at least 3.
We close with an appendix on various 3-dimensional results.

\section{Notation and Definitions}
\label{sec:notation}

Throughout, spaces are assumed to be metrizable, separable, and either compact or one-ended ($\R$ excepted).
Manifolds are assumed to be smooth, connected, and oriented.
Manifold boundaries are oriented by the standard outward normal first convention.
A manifold without boundary is \emph{closed} if it is compact and is \emph{open} if it is noncompact.
Write $A \approx B$ to mean $A$ is diffeomorphic to $B$ (not necessarily preserving orientation).
A map is \emph{proper} provided the inverse image of each compact set is compact.
A \emph{ray} is a proper embedding of $[0,\infty)$, where $[0,\infty) \subset \R$ is standardly oriented \cite[Ch.3]{gp}.

\begin{definition}[Connected Sum at Infinity]\label{def:CSI}
    Let $M$ and $N$ be open manifolds of the same dimension $m\geq 2$.  Fix rays $r\subset M$ and $r'\subset N$.  Form the \emph{connected sum at infinity} (CSI) of $(M,r)$ and $(N,r')$, denoted $(M,r) \csi (N,r')$, as follows.  Let $\nu r \subset M$ and $\nu r' \subset N$ be smooth, closed regular neighborhoods of $r$ and $r'$ respectively.  Identify $M - \Int \nu r$ and $N - \Int \nu r'$ along $\partial \nu r \approx \R^{m-1}$ and $\partial \nu r'\approx \R^{m-1}$ via an orientation reversing diffeomorphism $\phi$ as in Figure~\ref{fig:csi_sum}.
\end{definition}

\begin{figure}[htbp!]
    \centerline{\includegraphics[scale=1.0]{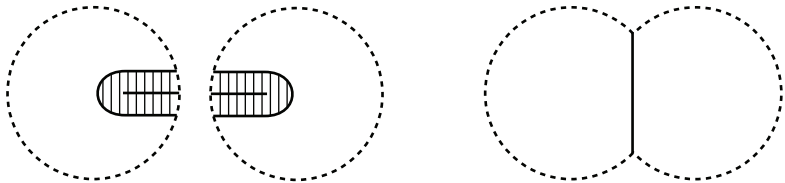}}
    \caption{CSI summands $(M,r)$ and $(N,r')$ with $\nu r$ and $\nu r'$ hatched (left). Result of CSI operation $(M,r)\csi (N,r')$ (right).}
\label{fig:csi_sum}
\end{figure}

\begin{remarks}
\noindent
\begin{enumerate}[label=(\arabic*),leftmargin=*]\setcounter{enumi}{0}
\item By a common abuse, we consider the manifold $(M,r)\csi (N,r')$ to be smooth (see Hirsch \cite[p.~184]{hirsch}).

\item Any diffeomorphism of $\R^{m-1}$ is isotopic to a linear mapping \cite[p.~34]{milnor}.
Together with uniqueness of regular neighborhoods \cite[\S3]{cks}, this shows that the diffeomorphism type of $(M,r) \csi (N,r')$ is independent of the choices of $\nu r$, $\nu r'$, and $\phi$.

\item The given definition of CSI is just sufficient for our purposes.  It is subsumed by a more general definition \cite{cks} that: (i) applies to differentiable, piecewise linear, and topological manifolds, (ii) yields a manifold/ray pair, (iii) is defined on any countable number of summands (see also Gompf \cite{gompf}), and (iv) is commutative and associative.
\end{enumerate}
\end{remarks}

We will use cohomology algebras at infinity to distinguish manifolds.
Just as cohomology is a homotopy invariant of spaces, the cohomology algebra at infinity is a proper homotopy invariant of spaces \cite[Ch.~3]{hr}.
Throughout, let $R$ be a commutative, unital ring.
If $X$ is any topological space, then we define the poset $(\mathcal{K},\leq)$ where $\mathcal{K}$ is the set of compact subsets of $X$ and $K\leq K'$ means $K\subseteq K'$.
We have a direct system of graded $R$-algebras $H^*(X - K;R)$, where $K\in\mathcal{K}$.
The morphisms of this direct system are restrictions.
Define $\ci{\ast}{X}{R}$, the \emph{cohomology algebra at infinity}, to be the direct limit of this system.
Similarly, we define $\rci{\ast}{X}{R}$ using reduced cohomology.

If $K_1 \subseteq K_2 \subseteq \cdots$ is a compact exhaustion of $X$, then we may compute $\ci{\ast}{X}{R}$ using the direct system indexed by the $K_j$. Namely,
\begin{equation}
    \ci{\ast}{X}{R} \cong \ilim_j H^*(X - K_j;R).
\end{equation}\label{eq:direct_limit}

We employ the standard explicit model of the direct limit where an element of $\ci{\ast}{X}{R}$ is represented by an element of $H^*(X - K; R)$ for some compact $K$.
Two representatives $\alpha \in H^*(X - K; R)$ and $\alpha' \in H^*(X - K'; R)$ are equivalent if they have the same restriction in some $H^*(X - K'';R)$, where $K,K'\subseteq K''$.

\section{Ladder Manifolds}
\label{sec:ladders}

In this section, we define ladder manifolds and compute their cohomology algebras at infinity.  
Ladder manifolds play a key role in our proof of the Main Theorem.
Fix closed manifolds $X$ and $Y$ of the same dimension $n\geq 2$.

\begin{definition}[Ladder Manifold]\label{def:ladder_manifold}
	The \emph{ladder manifold} of $X$ and $Y$, denoted $\la{X}{Y}$, is the oriented $(n+1)$-manifold obtained from the disjoint union
	\[
		([0,\infty)\times X) \sqcup ([0,\infty)\times Y)
	\]
	by performing countably many oriented $0$-surgeries as in Figure~\ref{fig:ladder_manifold}.
	\begin{figure}[htbp!]
    \centerline{\includegraphics[scale=1.0]{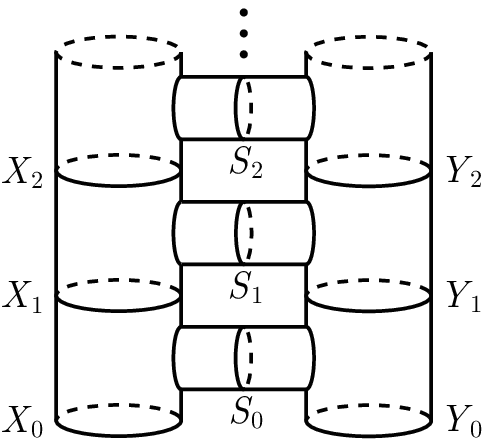}}
    \caption{Ladder manifold $\la{X}{Y}$.}
\label{fig:ladder_manifold}
\end{figure}
The manifolds \hbox{$[0,\infty) \times X$} and $[0,\infty) \times Y$ are the \emph{stringers}. Let $X_t:=\{t\}\times X$ and $Y_t := \{t\}\times Y$.  The glued-in copies of $D^1\times S^n$ are the \emph{rungs}, one for each integer $j\geq 0$. Let $S_j := \{0\}\times S^n$ be the central sphere in the $j$th rung.

	More explicitly, fix closed $n$-balls $B_X \subset X$ and $B'_X \subset \Int B_X$, and similarly for $Y$.  For each integer $j\geq 0$, perform an oriented $0$-surgery using $(n+1)$-disks, one in $\Int\pa{[j,j+1]\times B'_X}$ (see Figure~\ref{fig:surgery}) and the other in $\Int\pa{[j,j+1]\times B'_Y}$.

\begin{figure}[htbp!]
    \centerline{\includegraphics[scale=1.0]{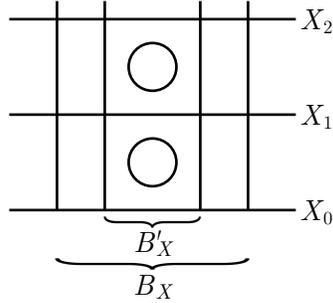}}
    \caption{Circles indicating disks in $[0,\infty)\times X$ used for $0$-surgeries.}
\label{fig:surgery}
\end{figure}

The ladder manifold $\la{X}{Y}$ is canonically oriented.
As $X$ and $Y$ are oriented, the stringers $[0,\infty) \times X$ and $[0,\infty)\times Y$ are each given the product orientation.
Note that $\la{X}{Y}$ has boundary $X_0 \sqcup Y_0$, oriented as $-X_0 - Y_0$ \cite[Ch.~3]{gp}.
Orient each $S_j$ so that the oriented boundary of the first cobordism in Figure~\ref{fig:cobordism} is $X_{j+1} - X_j + S_j$.  It follows that the oriented boundary of the second cobordism in Figure~\ref{fig:cobordism} is $Y_{j+1} - Y_j - S_j$.  This completes our description of the ladder manifold $\la{X}{Y}$.	
\end{definition}

\begin{figure}[htbp!]
    \centerline{\includegraphics[scale=1.0]{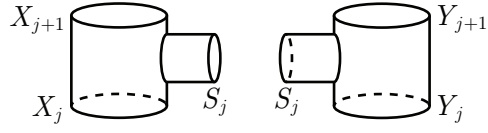}}
    \caption{Two oriented cobordisms in $\la{X}{Y}$.}
\label{fig:cobordism}
\end{figure}

The remainder of this section is devoted to computing the cohomology algebra at infinity of $\la{X}{Y}$.  For each integer $j\geq 0$, let $W_j$ be the submanifold of $\la{X}{Y}$ consisting of points of height $t\geq j$ (height is depicted vertically in Figure~\ref{fig:ladder_manifold}).
Note that $W_0=\la{X}{Y}$, $W_j\approx\la{X}{Y}$ for each $j$,
\[
	W_0 \supset W_1 \supset W_2 \supset \cdots
\]
and $\bigcap_j W_j = \emptyset$.
The inclusions $i_j \colon W_{j+1} \to W_{j}$ induce the direct system
\begin{equation}\label{eq:direct_system}
	\rc{\ast}{W_0}{R} \stackrel{i^*_0}{\longrightarrow} \rc{\ast}{W_1}{R} \stackrel{i^*_1}{\longrightarrow} \rc{\ast}{W_2}{R} \stackrel{i^*_2}{\longrightarrow} \cdots
\end{equation}
Evidently, 
\begin{equation}\label{eq:direct_limit_2}
	\rci{\ast}{\la{X}{Y}}{R} \cong \ilim_j \wt{H}^*(W_j;R).
\end{equation}

Let $J \subset \la{X}{Y}$ be the noncompact $n$-complex shown in Figure~\ref{fig:jacobs_ladder}.  It is an iterated wedge of $n$-spheres (the $S_j$'s from above), $1$-spheres (the $T_j$'s shown), and a 1-cell (bottom).  The complex $J$ is a variant of Jacob's ladder \cite[p.~25]{hr}.

\begin{figure}[htbp!]
    \centerline{\includegraphics[scale=1.0]{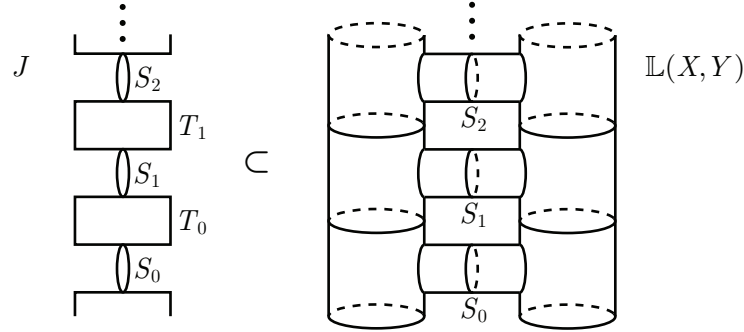}}
    \caption{One-ended $n$-complex $J\subset\la{X}{Y}$.}
\label{fig:jacobs_ladder}
\end{figure}

We remind the reader that $R$ denotes a commutative, unital ring.  Recall that $R[x] \cong \bigoplus_{n=0}^\infty R$ and $R[[x]] \cong \prod_{n=0}^\infty R$ as $R$-modules.  In general, $R[x]$ is a free $R$-module, but $R[[x]]$ need not be.  When $R$ is a field, $R[[x]]$ is an $R$-vector space and hence a free $R$-module.  However, $\Z[[x]]$ is \emph{not} a free $\Z$-module \cite{schroeer}.

The nonzero reduced integer homology groups of $J$ are $\wt{H}_n(J) \cong \Z[s]$ and $\wt{H}_1(J) \cong \Z[t]$, where $s^k$ corresponds to the fundamental class of $S_k$ and $t^k$ corresponds to the fundamental class of $T_k$.  By the Universal Coefficient Theorem, the nonzero reduced cohomology groups of $J$ are 
\begin{align*}
	\rc{n}{J}{R}	&\cong	\hom{\Z}{\Z[s]}{R} \cong R[[\sigma]]\\
	\rc{1}{J}{R}	&\cong	\hom{\Z}{\Z[t]}{R} \cong R[[\tau]]
\end{align*}
since the Ext terms vanish in all dimensions.  All cup products in $\rc{\ast}{J}{R}$ vanish.

For each integer $j\geq 0$, let $J_j$ denote the points in $J$ of height $t\geq j$.
Note that $J_0 = J$ and $J_j$ is homeomorphic to $J$ for each $j$.
Define
\[
	V_j := (X_j \vee J_j) \vee Y_j \subset W_j
\]
as shown in Figure~\ref{fig:retract} where $\iota_j \colon V_j \to W_j$ is inclusion.

\begin{figure}[htbp!]
    \centerline{\includegraphics[scale=1.0]{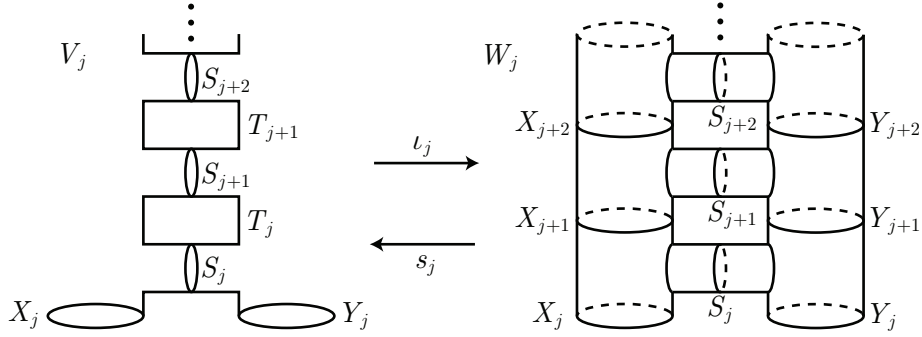}}
    \caption{Strong deformation retract $V_j$ of $W_j$.}
\label{fig:retract}
\end{figure}

\begin{lemma}\label{lem:sdr}
For each $j$, there is a strong deformation retraction $s_j\colon W_j \to V_j$.
\end{lemma}
\begin{proof}
	We begin by retracting the stringer portions of $W_j$, while fixing the rungs pointwise.  Figure~\ref{fig:sdr} shows schematically how do this above $X_j$; the same argument applies above $Y_j$.
\begin{figure}[h!]
\centering
\subfigure[\emph{Region above $X_j - B_X$ retracts to $X_j - B_X$, while region above $B_X-B'_X$ retracts to a hyperboloid.}]
{
    \label{fig:sdr_1}
    \includegraphics[scale=1]{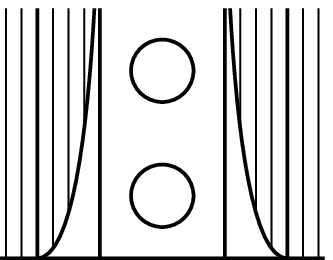}
}
\hspace{1cm}
\subfigure[\emph{Retraction of region under hyperboloid.}]
{
    \label{fig:sdr_2}
    \includegraphics[scale=1]{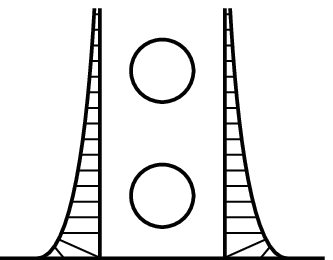}
}
\hspace{1cm}
\subfigure[\emph{Retraction of region above $B'_X$.}]
{
    \label{fig:sdr_3}
    \includegraphics[scale=1]{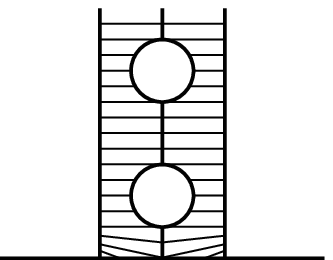}
}
\hspace{1cm}
\subfigure[\emph{Result of prior three retractions.}]
{
    \label{fig:sdr_4}
    \includegraphics[scale=1]{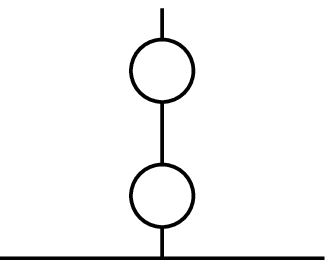}
}
\caption{Strong deformation retraction of portion of $W_j$ above $X_j$ to the iterated wedge of $X_j$ and an infinite string of $n$-spheres and intervals.  Rungs of $W_j$ are fixed pointwise at all times.}
\label{fig:sdr}
\end{figure}	
Next, simultaneously retract the remaining rung portions as shown in Figure~\ref{fig:rung_retract}.
\begin{figure}[htbp!]
    \centerline{\includegraphics[scale=1.0]{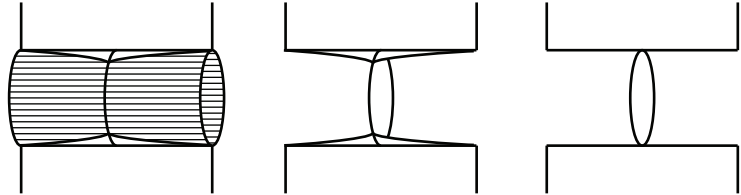}}
    \caption{Strong deformation retraction of a rung.}
\label{fig:rung_retract}
\end{figure}
This completes our description of $s_j$.
\end{proof}

We have the following diagram where the left maps are the obvious inclusions and projections.
\begin{equation}\label{eq:cd_0}\begin{split}
\xymatrix@R=12pt{
	X_j \ar@<.5ex>[dr] & &\\
	J_j \ar@<.5ex>[r] & \, V_j \ar@<.5ex>[r]^{\iota_j}  \ar@<.5ex>[ul] \ar@<.5ex>[l] \ar@<.5ex>[dl] & 
	W_j \ar@<.5ex>[l]^{s_j}\\
	Y_j \ar@<.5ex>[ur] & &
	  }
\end{split}
\end{equation}

\begin{corollary}\label{cor:W_j_cohomology}
	For each $j$, there is an isomorphism of graded $R$-algebras
\[
	\rc{\ast}{W_j}{R} \cong \rc{\ast}{X}{R} \oplus \rc{\ast}{J_j}{R} \oplus \rc{\ast}{Y}{R}.
\]
The cup product is coordinatewise in the direct sum.
\end{corollary}
\begin{proof}
	This follows immediately from Lemma~\ref{lem:sdr}, diagram~\eqref{eq:cd_0}, and the computation of the cohomology algebra of a wedge sum \cite[p.~215]{hatcher}.
\end{proof}

Recall the direct system \eqref{eq:direct_system}.
\begin{lemma}\label{lem:surjective}
	Each $i^*_j$ is surjective.
\end{lemma}
\begin{proof}
For each $j\geq 0$, there is an obvious retraction $r_j \colon W_j \to W_{j+1}$.
It sends $X_{j+t}$ to $X_{j+2-t}$ and $Y_{j+t}$ to $Y_{j+2-t}$ for $t\in [0,1]$ and sends the bottom rung of $W_j$ to the bottom rung of $W_{j+1}$.
Thus, $r_j\circ i_j = \tn{id}$, and so $(r_j\circ i_j)^* = i^*_j \circ r^*_j = \tn{id}^*$.
As $\tn{id}^*$ is an isomorphism on $\rc{\ast}{W_{j+1}}{R}$, $i^*_j$ is surjective.
\end{proof}

For $j < k$, let $i^*_{k,j}$ be the composition $i^*_k \circ i^*_{k-1} \circ \dots \circ i^*_j$.
By Lemma~\ref{lem:surjective}, each element in the direct limit \eqref{eq:direct_limit_2} has a representative in $\rc{\ast}{W_0}{R}$.
Indeed, if $\alpha \in \rc{\ast}{W_j}{R}$ represents an element $\omega$ in the direct limit, then there exists some $\beta \in \rc{\ast}{W_0}{R}$ such that $i^*_{j-1,0}(\beta) = \alpha$, so $\beta$ also represents $\omega$.
Thus, we can write
\begin{equation}\label{eq:mod_twiddle}
	\rci{\ast}{\la{X}{Y}}{R} \cong \ilim_j \rc{\ast}{W_j}{R} \cong \rc{\ast}{W_0}{R}/ \sim
\end{equation}
where $\alpha \sim \beta$ if and only if there exists $j$ such that $i^*_{j,0}(\alpha) = i^*_{j,0}(\beta)$.

\begin{proposition}\label{prop:ladder_algebra}
The cohomology algebra at infinity of $\la{X}{Y}$ is
\[
	\rci{k}{\la{X}{Y}}{R} \cong
	\begin{cases}
		(\rc{n}{X}{R} \oplus R[[\sigma]] \oplus \rc{n}{Y}{R}) / K &\tn{if $k = n$,}\\
		\rc{k}{X}{R} \oplus 0 \oplus \rc{k}{Y}{R} &\tn{if $2\leq k \leq n-1$,}\\
		\rc{1}{X}{R} \oplus R[[\tau]]/R[\tau] \oplus \rc{1}{Y}{R} &\tn{if $k = 1$,}\\
		0 &\tn{otherwise,}
	\end{cases}
\]
where $K := \{(\sum \beta_i, \beta, -\sum \beta_i) \colon \beta = \sum \beta_i\sigma^i \in R[\sigma]\} \cong R[\sigma]$. The cup product is coordinatewise in the direct sum.  
\end{proposition}

\begin{proof}
By the preceding discussion, it remains to describe $\sim$.
As $$\rc{\ast}{W_j,W_{j+1}}{R} \cong H^{\ast}(W_j/W_{j+1};R),$$ the long exact sequence for the pair $(W_j,W_{j+1})$ along with Lemma~\ref{lem:surjective} implies that $i^*_j \colon \rc{k}{W_j}{R} \to \rc{k}{W_{j+1}}{R}$ is an isomorphism for $2\leq k \leq n-1$.
Thus, $\rci{k}{\la{X}{Y}}{R} \cong \rc{k}{W_0}{R}$ for $2\leq k \leq n-1$.

For each $j$, we have the following commutative diagram of spaces.
\begin{equation}\label{eq:cd}\begin{split}
\xymatrix@C=6pt{
	  & W_{j+1} \ar@{->}[rrrr]^{i_j} & & & & W_j \ar@{->}[d]^{s_j} &\\
	  & V_{j+1}	\ar@{->}[u]^{\iota_{j+1}} \ar@{->}[rrrr]^{d_j} & & & & V_j \ar@{->}[dl]^{} \ar@{->}[d]^{} \ar@{->}[dr]^{} &\\
	X_{j+1} \ar@{->}[ur]^{} & J_{j+1}\ar@{->}[u]^{} & Y_{j+1}\ar@{->}[ul]^{} & & X_j & J_j & Y_j
}
\end{split}
\end{equation}
Here, $s_j$ is the retraction from Lemma~\ref{lem:sdr}, $\iota_{j+1}$ is inclusion, and  $d_j := s_j\circ i_j \circ \iota_{j+1}$.  The bottom maps are again the obvious inclusions and projections.

There are nine compositions in \eqref{eq:cd} that begin and end on the bottom row.  The geometry of these compositions is straightforward given our definition of $s_j$.  Recalling our orientation conventions, these compositions induce a homomorphism on integer homology $$\psi\colon \wt{H}_{\ast}(X) \oplus \wt{H}_{\ast}(J_{j+1}) \oplus \wt{H}_{\ast}(Y) \to \wt{H}_{\ast}(X) \oplus \wt{H}_{\ast}(J_j) \oplus \wt{H}_{\ast}(Y).$$
In dimensions other than $n$, $\psi = \tn{id} \oplus \tn{inclusion} \oplus \tn{id}$.
In dimension $n$, $\psi$ is given by $(a,b,c) \mapsto (a, (c-a)s^j+b,c)$ for $b\in s^{j+1}\Z[s]$.

As $X$ and $Y$ are closed and oriented $n$-manifolds,
\[
	\rh{n-1}{W_j}{\Z} \cong 
	\begin{cases}
		\rh{n-1}{X}{\Z}\oplus 0 \oplus \rh{n-1}{Y}{\Z} &\tn{if $n>2$}\\
		\rh{n-1}{X}{\Z}\oplus t^j\Z[t] \oplus \rh{n-1}{Y}{\Z} &\tn{if $n=2$}
	\end{cases}
\]
is free abelian. Thus, $\tn{Ext}(\rh{m-1}{W_j}{\Z},R) = 0$ for $m=1$ and $m=n-1$.
So, the Universal Coefficient Theorem implies that $i^*_j$ is the dual homomorphism of $\psi$ in these dimensions.  Therefore, $i^*_j$ is clear in dimension 1.  In dimension $n$, $i^*_j$ sends $(\alpha,\beta,\gamma)$ to $(\alpha-\beta_j,\beta-\beta_j\sigma^j,\gamma+\beta_j)$ where $\beta\in \sigma^jR[[\sigma]]$.
To determine $\sim$ in the remaining dimensions 1 and $n$, it suffices to describe the subgroups of $\rc{1}{W_0}{R}$ and $\rc{n}{W_0}{R}$ consisting of elements that are sent to 0 by some $i^*_{j,0}$.  By our description of $i^*_j$ as the dual of $\psi$, these subgroups are exactly $0\oplus R[\tau] \oplus 0$ in dimension 1 and $K$ in dimension $n$.

The cup product structure of the algebra $\rci{\ast}{\la{X}{Y}}{R}$ can be summarized as $[\alpha]\cup[\beta] = [\alpha\cup\beta]$ for $\alpha, \beta \in \rc{\ast}{W_0}{R}$.  This is a direct consequence of the definition of a direct limit of algebras and the fact that every element of $\rci{\ast}{\la{X}{Y}}{R}$ has a representative in $\rc{\ast}{W_0}{R}$.  Combining these facts yields the claim about cup products in the statement of this proposition.
\end{proof}

\begin{remark}\label{rmk:injection}
	Observe that the canonical map
	\[
		\pi\colon \rc{n}{X}{R} \oplus 0 \oplus \rc{n}{Y}{R} \to (\rc{n}{X}{R} \oplus R[[\sigma]] \oplus \rc{n}{Y}{R}) / K
	\]
	is injective.  In particular, if $R$ is a field, then the image of $\pi$ is a 2-dimensional $R$-vector subspace of an uncountably-infinite dimensional $R$-vector space.
\end{remark}

\section{Stringer Sum}
\label{sec:stringer_sum}
We now define stringer sum, an operation on a ladder manifold and a disjoint stringer that yields a new ladder manifold.
Recall that a manifold $A$ is \emph{neatly} embedded in a manifold $B$ if $A \cap \partial B = \partial A$ and this intersection is transverse.
A \emph{straight} ray in $[0,\infty) \times A$ has the form $[0,\infty)\times \{a\}$.

\begin{definition}[Stringer Sum]\label{def:stringer_sum}
	Fix a ladder manifold $\la{X}{Y}$ and a disjoint stringer $[0,\infty) \times Z$ of the same dimension $n+1 \geq 3$. 
Fix neatly embedded, straight rays $r \subset [0,\infty)\times Z$ and $r' \subset \la{X}{Y}$, where $r'$ lies in one of the stringers of $\la{X}{Y}$ and avoids $B_X$ and $B_Y$ (see Figure~\ref{fig:stringer_sum}).
	Define the \emph{stringer sum} of $(\la{X}{Y},r')$ and $([0,\infty) \times Z, r)$, denoted $$(\la{X}{Y},r') \cdot ([0,\infty) \times Z,r),$$ as follows.
	Let $\nu r \subset [0,\infty)\times Z$ and $\nu r' \subset \la{X}{Y}$ be normal, closed tubular neighborhoods of $r$ and $r'$ respectively.
	Identify $([0,\infty) \times Z) - \Int \nu r$ and $\la{X}{Y} - \Int \nu r'$ along $\partial \nu r$ and $\partial \nu r'$ via an orientation reversing, fiber respecting diffeomorphism so that connected sum is achieved at each height.
\end{definition}

\begin{figure}[htbp!]
    \centerline{\includegraphics[scale=1.0]{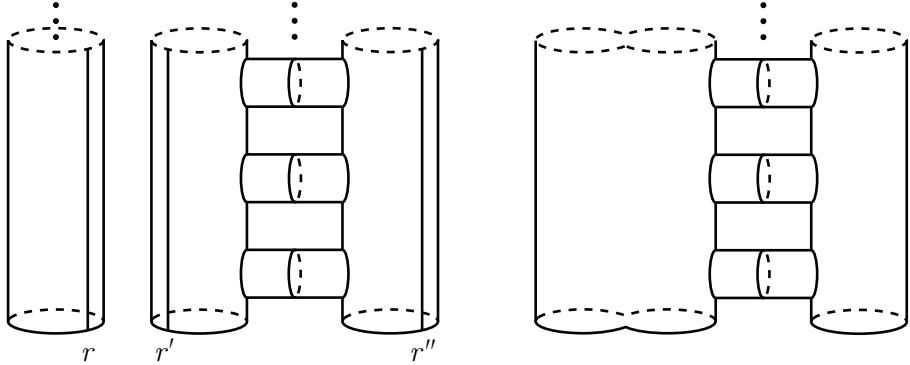}}
    \caption{Stringer $[0,\infty) \times Z$ with ray $r$ and ladder manifold $\la{X}{Y}$ with rays $r'$ and $r''$, one in each stringer (left). Stringer sum $(\la{X}{Y},r') \cdot ([0,\infty) \times Z,r)$ (right).}
\label{fig:stringer_sum}
\end{figure}

With notation as in Figure~\ref{fig:stringer_sum}, observe that 
\[
	(\la{X}{Y},r') \cdot ([0,\infty) \times Z,r) \approx \la{X\cs Z}{Y}
\]
and
\[
	(\la{X}{Y},r'') \cdot ([0,\infty) \times Z,r) \approx \la{X}{Y\cs Z}
\]
where $\cs$ denotes oriented connected sum CS.


\section{Examples: Lens Spaces}
\label{sec:lens}
For each positive integer $k$, let $L_k$ denote the 3-dimensional lens space $L(k,1)$, which is obtained by performing $-k$-surgery on the unknot in $S^3$~\cite[p.~158]{gs}.
Recall that
\[
	\rh{i}{L_k}{\Z} \cong
	\begin{cases}
		\Z &\tn{if $i=3$,}\\
		\Z_k &\tn{if $i=1$,}\\
		0 &\tn{otherwise,}
	\end{cases}
\]
\[
	\rc{i}{L_k}{\Z} \cong
	\begin{cases}
		\Z &\tn{if $i=3$,}\\
		\Z_k &\tn{if $i=2$,}\\
		0 &\tn{otherwise,}
	\end{cases}
\]
and $\rc{i}{L_k}{\Z_k} \cong \Z_k$ for $i = 1,2,3$.

For computability, we now restrict to field coefficients.
Fix a prime $p>0$.
By Poincar\'{e} duality, if $M$ is a closed $m$-manifold and $0 \neq \alpha \in \rc{i}{M}{\Z_p}$ where $i<m$, then there exists $\beta\in \rc{m-i}{M}{\Z_p}$ such that $\alpha \cup \beta \in \rc{m}{M}{\Z_p}$ is a generator \cite[p.~250]{hatcher}.

If $M$ is an $(n+1)$-manifold, then define 
\[
\Gamma_p\pa{M} := \langle \alpha \cup \beta \mid \deg \alpha, \deg \beta < n \rangle \leq \rci{n}{M}{\Z_p}
\]
to be the vector subspace of $\rci{n}{M}{\Z_p}$ generated by products of classes of degree less than $n$ in $\rci{\ast}{M}{\Z_p}$.
The dimension of $\Gamma_p$ as a $\Z_p$-vector space, denoted $\dim_{\Z_p}\Gamma_p$, is a graded algebra invariant since an isomorphism of graded algebras respects products and gradings.

By Proposition~\ref{prop:ladder_algebra} and Remark~\ref{rmk:injection}, we have
\[
	\dim_{\Z_p}\Gamma_p\pa{\la{X}{Y}} \leq 2
\]
for any ladder manifold.

\begin{proposition}\label{prop:lens_ladder}
The stringer sums
\[
	\pa{\la{L_p}{S^3},r'} \cdot ([0,\infty) \times L_p,r) \approx \la{L_p \cs L_p}{S^3}
\] 
and
\[
	\pa{\la{L_p}{S^3},r''} \cdot ([0,\infty) \times L_p,r) \approx \la{L_p}{L_p}
\]
have nonisomorphic $\Z_p$-cohomology algebras at infinity.
\end{proposition}

\begin{proof}
Consider the algebras $\rci{\ast}{\la{L_p \cs L_p}{S^3}}{\Z_p}$ and $\rci{\ast}{\la{L_p}{L_p}}{\Z_p}$, computed as in Proposition~\ref{prop:ladder_algebra}.  Notice that $$\dim_{\Z_p} \Gamma_p\pa{\la{L_p \cs L_p}{S^3}} = 1$$ and
\[
	\dim_{\Z_p} \Gamma_p\pa{\la{L_p}{L_p}} = 2. \qedhere
\]
\end{proof}

\section{Proof of the Main Theorem}
\label{sec:proof}

For each integer $k\geq 1$, let $E_k$ be the $D^2$ bundle over $S^2$ with Euler number $-k$, which is a $0$-handle union a $2$-handle attached along a $-k$ framed unknot \cite[pp.~119--120]{gs}.
Note that $\partial E_k = L_k$.
Define  $$Y_k := \la{L_k}{S^3}\cup_\partial E_k \cup_\partial D^4$$ and $$Z_k := ([0,\infty) \times L_k) \cup_\partial E_k \approx \Int E_k.$$  Both $Y_k$ and $Z_k$ are smooth, open, one-ended $4$-manifolds.  We refer to $Y_k$ as a \emph{capped ladder} and $Z_k$ as a \emph{capped stringer} (see Figure~\ref{fig:csi_capped}).
\begin{figure}[htbp!]
    \centerline{\includegraphics[scale=1.0]{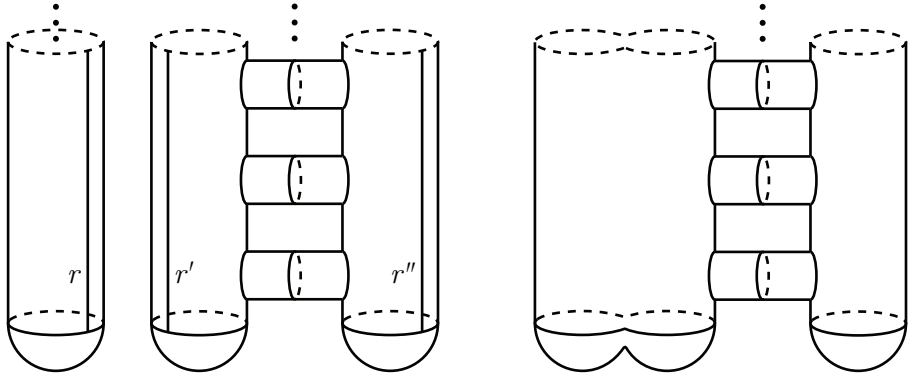}}
    \caption{Capped stringer $Z_k$, capped ladder $Y_k$, and result of CSI operation $(Y_k,r')\csi (Z_k,r)$.}
\label{fig:csi_capped}
\end{figure}
Let $r',r'' \subset \la{L_k}{S^3} \subset Y_k$ and $r\subset [0,\infty) \times L_k \subset Z_k$ be rays as in Proposition~\ref{prop:lens_ladder} and Figure~\ref{fig:csi_capped}.

\begin{theorem}\label{thm:main_examples}
	Let $p>0$ be prime. The manifolds $$M_1 := (Y_p,r') \csi (Z_p,r)$$ and $$M_2 := (Y_p,r'') \csi (Z_p,r)$$ are not proper homotopy equivalent.
\end{theorem}
\begin{proof}
	First, note that
	\[
		M_1 \approx \la{L_p\cs L_p}{S^3} \cup_\partial (E_p \csb E_p) \cup_\partial D^4
	\]
	and
	\[
		M_2 \approx \la{L_p}{L_p} \cup_\partial E_p \cup_\partial E_p,
	\]
	where $\csb$ denotes oriented connected sum boundary CSB.
	Thus,
	\[
		\rci{\ast}{M_1}{\Z_p} \cong \rci{\ast}{\la{L_p\cs L_p}{S^3}}{\Z_p}
	\]
	and
	\[
		\rci{\ast}{M_2}{\Z_p} \cong \rci{\ast}{\la{L_p}{L_p}}{\Z_p}.
	\]
	By Proposition~\ref{prop:lens_ladder}, $\rci{\ast}{M_1}{\Z_p}$ and $\rci{\ast}{M_2}{\Z_p}$ are not isomorphic.  Therefore, $M_1$ and $M_2$ are not proper homotopy equivalent.
\end{proof}

\begin{remark}\label{rmk:distinct}
Observe that $Y_j \not\approx Y_k$ and $Z_j \not\approx Z_k$ for positive integers $j\neq k$, and $Y_j \not\approx Z_k$ for any positive integers $j$ and $k$.  These observations hold by the following facts: (i) $\rci{2}{Y_j}{\Z} \cong \Z_j$, (ii) $\rci{2}{Z_j}{\Z} \cong \Z_j$, (iii) $\rci{1}{Y_j}{\Z} \cong \Z[[\tau]]/\Z[\tau]$, and (iv) $\rci{1}{Z_k}{\Z} = 0$.
Facts (i) and (iii) follow from Proposition~\ref{prop:ladder_algebra}, while (ii) and (iv) follow from the basic property that $\rci{\ast}{[0,\infty) \times X}{R} \cong \rc{\ast}{X}{R}$ for each closed manifold $X$.
Hence, Theorem~\ref{thm:main_examples} implies the Main Theorem as there are infinitely many primes.
\end{remark}

\section{Generalizations of the Main Examples}
\label{sec:generalization}

Our main examples from Section~\ref{sec:proof} are readily modified to produce more $4$-dimensional examples as well as others of all dimensions at least 3.  Define $$T^k := \underbrace{S^1 \times \dots \times S^1}_{k}.$$

\begin{enumerate}[label=(\arabic*),leftmargin=*]\setcounter{enumi}{0}
\item In $Y_p$, $S^3$ can be replaced with any $\Z_p$-homology $3$-sphere $\Sigma^3$ and $D^4$ with any smooth null-cobordism of $\Sigma^3$.

\item For any $j\geq 1$, $Z_p$ can be replaced with $Z_{jp}$, since $\rc{i}{L_{jp}}{\Z_p} \cong \Z_{p}$ for $i=1,2,3$.

\item To obtain examples in all dimensions $n+1\geq 4$, replace $L_p$ with $L_p \times T^{n-3}$ in both $Y_p$ and $Z_p$, and replace $S^3$ with $S^n$.  Cap with $E_p \times T^{n-3}$ and $D^{n+1}$.  Crossing with $S^1$ does not affect $\dim_{\Z_p}\Gamma_p$.  We obtain infinitely many examples this way by the following observations derived from Remark~\ref{rmk:distinct} and the K\"{u}nneth formula
\begin{align*}
	\tn{torsion}\, \rci{2}{Y_p}{\Z} &\cong \Z_p\\
	\tn{torsion}\, \rci{2}{Z_p}{\Z} &\cong \Z_p\\
	\rci{1}{Y_p}{\Z} &\cong \Z^{n-3} \oplus \Z[[\tau]]/\Z[\tau]\\
	\rci{1}{Z_p}{\Z} &\cong \Z^{n-3}
\end{align*}

\item Let $\Sigma_g$ be the closed surface of genus $g$. Let $H_g$ be the $3$-dimensional handlebody with $\partial H_g = \Sigma_g$.  Define $$Y_g := \la{\Sigma_g}{S^2} \cup_\partial H_g \cup_\partial D^3$$ and $$Z_g := ([0,\infty)\times \Sigma_g) \cup_\partial H_g.$$
Let $r',r'' \subset Y_g$ and $r\subset Z_g$ be straight rays as in Figure~\ref{fig:csi_capped}.
For $g,h\in \Z^+$, define
\[
	M_1(g,h) := (Y_g,r') \csi (Z_h,r)
\]
and
\[
	M_2(g,h) := (Y_g,r'') \csi (Z_h,r).
\]
Fix any prime $p>0$.  By Proposition~\ref{prop:ladder_algebra}, $$\dim_{\Z_p} \Gamma_p(M_1(g,h)) = 1$$ and $$\dim_{\Z_p} \Gamma_p(M_2(g,h)) = 2.$$
Thus, $M_1(g,h)$ and $M_2(g,h)$ are not proper homotopy equivalent.

We can prove the Main Theorem for $3$-manifolds using the collection of pairs $Y_1$ and $Z_g$ for $g\in \Z^+$.  These manifolds are distinguished by the following facts 
\begin{align*}
	\rci{1}{Y_1}{\Z} &\cong \Z^2 \oplus \Z[[\tau]]/\Z[\tau]\\
	\rci{1}{Z_g}{\Z} &\cong \Z^{2g}
\end{align*}
Interestingly, we cannot distinguish $Y_g$ and $Y_{g'}$ for any $g\neq g'$.

We obtain infinitely many new examples in all dimensions $n+1\geq 4$ by considering ladders and stringers based on $\Sigma_g \times T^{n-2}$ and $S^n$.  In these dimensions, one can distinguish all of the summands.  Details are left to the interested reader.

\item\label{itm:two_ladders} Let $p,q>0$ be distinct primes.  Define
\[
	M(p,q) := \la{L_p}{L_q} \cup_\partial E_p \cup_\partial E_q.
\]
Fix neat, straight rays $r',r'' \subset \la{L_p}{L_q}$, one in each stringer.  The same technique used to prove Proposition~\ref{prop:ladder_algebra} yields the following table where $M := M(p,q)$.
\begin{table}[h!]\renewcommand{\arraystretch}{1.2}
\begin{center}
\begin{tabular}{|c|c|c|}\hline
	& $\dim_{\Z_p} \Gamma_p$ & $\dim_{\Z_q} \Gamma_q$\\ \hline
	$(M,r') \csi (M,r')$ & 1 & 2 \\ \hline
	$(M,r'') \csi (M,r'')$ & 2 & 1\\ \hline
	$(M,r') \csi (M,r'')$ & 2 & 2\\ \hline
\end{tabular}
\end{center}
\end{table}
Thus, the CSI of $M$ with itself yields at least 3 distinct manifolds up to proper homotopy.  We obtain infinitely many $M$ with this property since 
\[
	\rci{2}{M(p,q)}{\Z} \cong \Z_p \oplus \Z_q.
\]

\item\label{itm:many_ladders} Given $3$-dimensional lens spaces $L_{p_1}, L_{p_2}, \dots, L_{p_{m+1}}$, we define the \emph{generalized capped ladder manifold} $M\pa{p_1,p_2,\dots,p_{m+1}}$ inductively as follows. 
\begin{multline*}
	M\pa{p_1,p_2,\dots,p_{m+1}} := \\ \pa{M\pa{p_1,p_2,\dots,p_m},r} \csi \pa{\la{S^3}{L_{p_{m+1}}} \cup_\partial (D^4 \sqcup E_{p_{m+1}}),r'}
\end{multline*}
where $r$ and $r'$ are neat, straight rays in the $L_{p_m}$ and $S^3$ stringers respectively.
Now, fix $p_1, p_2, \ldots, p_m > 0$ to be distinct primes.  A similar calculation to the one in the previous item shows that the CSI of $M\pa{p_1,p_2,\dots,p_m}$ with itself yields at least $m+1$ distinct manifolds up to proper homotopy.  We obtain infinitely many $M$ with this property since
\[
	\rci{2}{M(p_1,p_2,\dots,p_m)}{\Z} \cong \bigoplus_{i=1}^m \Z_{p_i}.
\]

\item\label{itm:infinite_ladders} Fix $p_1, p_2, \dots $ to be distinct positive primes.  Define $M\pa{p_1,p_2,\dots}$ in analogy with the previous item.  Similarly, one may verify that the CSI of $M\pa{p_1,p_2,\dots}$ with itself yields countably infinitely many distinct manifolds up to proper homotopy.  We obtain infinitely many $M$ with this property since
\[
	\rci{2}{M(p_1,p_2,\dots)}{\Z} \cong \prod_{i=1}^\infty \Z_{p_i}.
\]
\end{enumerate}

\begin{remark}\label{rmk:Henry_King}
    In items \ref{itm:two_ladders}--\ref{itm:infinite_ladders}, we computed the cohomology algebras at infinity of generalized ladder manifolds using the geometric techniques of Section~\ref{sec:ladders}.
\end{remark}

\section{Appendix: CSI and $\R^3$}
\label{sec:appendix}

The purpose of this appendix is to give a simple proof in the smooth category that a CSI of $\R^3$ with itself need not yield $\R^3$.

Fix a smooth, proper embedding $f\colon\R^2\to\R^3$.
Let $H:=\im f$, a hyperplane in $\R^3$.
Let $A$ and $B$ denote the closures in $\R^3$ of the two components of $\R^3-H$.
So, $\partial A=\partial B=H$, $A\cap B=H$, and $A\cup B=\R^3$.
As we are interested in $H$ up to ambient isotopy of $\R^3$, we assume $f(0)=0$.

\begin{definition}[Nice $2$-disk]\label{def:nice_disk}
        A $2$-disk $D\subset\R^3$ is \emph{nice} provided: (i) $D$ is neatly embedded in $A$ or in $B$, and (ii) $\partial D$ is essential in $H-\cpa{0}$.
\end{definition}

\begin{lemma}\label{lem:nice_disk}
        Let $K\subset \R^3$ be compact. Then, there exists a nice $2$-disk $D\subset \R^3 -K$.
\end{lemma}

\begin{proof}
Let $D^3\subset\R^3$ be a $3$-disk centered at $0$ and containing $K$.
By replacing $K$ with $D^3$, we may assume $K$ is connected.
Let $B^2\subset \R^2$ be a disk centered at 0 containing $f^{-1}(K)$.
Let $K':= K\cup f(B^2)$, which is compact and connected.
Let $S\subset \R^3$ be a $2$-sphere such that $K'$ lies inside $S$, and $S$ meets $H$ transversely.
So, $S\cap H$ is a finite disjoint union of circles disjoint from $K'$, at least one of which is essential in $H-\cpa{0}$.
If there exist components of $S\cap H$ that are inessential in $H-\cpa{0}$,
then let $C$ be one that is innermost in $H-\cpa{0}$.
Then, $C$ bounds $2$-disks $\Delta\subset H-\cpa{0}$ and $D_1,D_2\subset S$ (see Figure~\ref{fig:Delta}).
Note that $\Delta$ is disjoint from $K'$.
\begin{figure}[htbp!]
    \centerline{\includegraphics[scale=1.0]{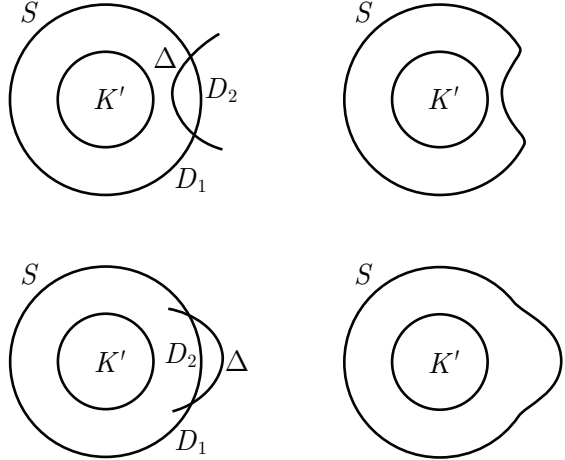}}
    \caption{Two possibilities: $\Delta$ inside or outside $S$ (left) and resulting $2$-sphere $S$ after isotopy across $B_2$ (right).}
\label{fig:Delta}
\end{figure}
Each of $\Delta\cup D_1$ and $\Delta\cup D_2$ is a (piecewise smooth) embedded $2$-sphere in $\R^3$.
Let $B_1$ and $B_2$ be the $3$-disks in $\R^3$ with boundaries $\Delta\cup D_1$ and $\Delta\cup D_2$ respectively.
As $K'$ is connected, $K'\subset\Int B_1$ or $K'\subset\Int B_2$, but not both.
Without loss of generality, assume $K'$ lies in $\Int B_1$.
Using $B_2$, isotop $D_2$ past $\Delta$ to a parallel copy of $\Delta$.
The hyperplane $H$ is fixed in the background during this isotopy of $S$.
The isotoped sphere is again called $S$.
Note that $K'$ remains inside $S$ and $C$ has been eliminated from $S\cap H$.
Repeat this procedure until all components of $S\cap H$ are essential in $H-\cpa{0}$.
Now, let $C$ be a component of $S\cap H$ that bounds a disk $D\subset S$ disjoint from the other components of $S\cap H$.
The disk $D$ is nice.
\end{proof}

Let $\R^3_+$ denote closed upper half-space.

\begin{lemma}\label{lem:AB_std}
        Either $A\approx\R^3_+$ or $B\approx\R^3_+$. 
\end{lemma}

\begin{proof}
Use Lemma~\ref{lem:nice_disk} to obtain a proper, disjoint collection $D_k$, $k\in\Z^{+}$, of
nice 2-disks.
As each $D_k$ lies in $A$ or in $B$, we may assume, without loss of generality, that infinitely
many $D_k$ lie in $A$.
\begin{figure}[htbp!]
    \centerline{\includegraphics[scale=1.0]{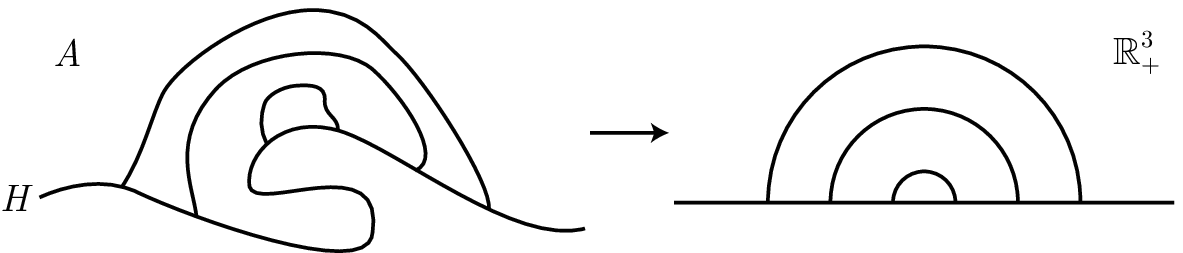}}
    \caption{Diffeomorphism $A \to \R^3_+$.}
\label{fig:exercise}
\end{figure}
To build a diffeomorphism $A\to\R^3_+$, proceed as indicated in Figure~\ref{fig:exercise} using repeatedly these tools: (i) the $2$- and $3$-dimensional smooth Schoenflies theorems \cite[Thm.~1.1]{hatcher3d},\cite[Ch.~III]{cerf}, and (ii) the fact that every diffeomorphism of $S^2$ extends to the $3$-disk \cite{munkres,smale} (see also \cite[Thm.~3.10.11]{thurston}).
\end{proof}

\begin{lemma}\label{lem:A_reg_neighb}
	Let $r\subset H$ be a ray.  Then, $A$ or $B$ is ambient isotopic to a smooth regular neighborhood
	 of $r$ in $\R^3$. In particular, the ambient isotopy class of $H$ in $\R^3$ is determined by the ambient isotopy class of $r$ in $\R^3$ and conversely.
\end{lemma}
\begin{proof}
By Lemma~\ref{lem:AB_std}, either $A$ or $B$ is diffeomorphic to $\R^3_+$.  Without loss of generality, let $g\colon A \to \R^3_+$ be a diffeomorphism.  Let $s \subset \Int \R^3_+$ be any straight ray, so $\R^3_+$ is a smooth regular neighborhood of $s$ in $\R^3$. As $g(r) \subset \R^2 \times \{0\}$ is necessarily unknotted \cite[p.~1845]{cks}, $s$ is ambient isotopic to $g(r)$ in $\R^3$. It follows that $A$ is a smooth regular neighborhood of $g^{-1}(s)$ in $\R^3$ and $r$ is ambient isotopic to $g^{-1}(s)$ in $\R^3$.  This proves the first claim in the lemma.  The second claim now follows by uniqueness of regular neighborhoods \cite[\S3]{cks}.
\end{proof}

\begin{proposition}\label{prop:hyperplane_ray}
        Let $r\subset H$ be any ray. The following are equivalent:
\begin{enumerate}[label=(\arabic*)]\setcounter{enumi}{0}
\item\label{eq:H_std} $H$ is unknotted in $\R^3$.
\item\label{eq:r_std} $r$ is unknotted in $\R^3$.
\item\label{eq:AB_Rp} $A\approx\R^3_+$ and $B\approx\R^3_+$.
\end{enumerate}    
\end{proposition}
\begin{proof}
\ref{eq:H_std} $\Leftrightarrow$ \ref{eq:r_std} by Lemma~\ref{lem:A_reg_neighb}. \ref{eq:H_std} $\Rightarrow$ \ref{eq:AB_Rp} is obvious.  For \ref{eq:AB_Rp} $\Rightarrow$ \ref{eq:H_std}, the hypotheses give orientation preserving diffeomorphisms $g\colon A \to \R^3_+$ and $h\colon B \to \R^3_-$.  Identify $\R^2 \times \{0\}$ with $\R^2$.
Define the diffeomorphism
\[
	\psi := \left. g \circ h^{-1}\right| \colon \R^2 \to \R^2.
\]
So, the diffeomorphism
\[
	k := (\psi \times \tn{id})\circ h \colon B \to \R^3_-
\]
satisfies $\left. k\right|H = \left. g\right|H$.
Thus, we have a homeomorphism $\mu \colon \R^3 \to \R^3$ where $\left. \mu\right|A = g$ and $\left. \mu\right|B = k$ are diffeomorphisms.
By standard collaring results in Hirsch \cite[p.~184]{hirsch}, we may assume $\mu$ is a diffeomorphism and $\mu(H) = \R^2$.
The result follows since $\mu$ is isotopic to the identity by Milnor \cite[p.~34]{milnor}.
\end{proof}

\begin{example}\label{ex:example}
Let $s\subset \R^3$ be a straight ray and let $r\subset \R^3$ be a knotted ray \cite[p.~983]{foxartin}.
Clearly $(\R^3,s) \csi (\R^3,s) \approx \R^3$.
Consider
\[
	M := (\R^3,r) \csi (\R^3,r).
\]
As $r$ is knotted, $\R^3 - \Int \nu r \not\approx \R^3_+$ by Proposition~\ref{prop:hyperplane_ray}. Let $H\subset M$ be the hyperplane determined by $\partial\nu r$.  By Lemma~\ref{lem:AB_std}, $M \not\approx \R^3$.
\end{example}

\begin{remarks}
\noindent
\begin{enumerate}[label=(\arabic*),leftmargin=*]\setcounter{enumi}{0}
\item Lemma~\ref{lem:AB_std} was proved by Harrold and Moise (1953) \cite{harroldmoise} in the piecewise linear category (see also Sikkema \cite{sikkema}).  This lemma can also be deduced in the topological (locally flat) category as follows.  Consider the 2-sphere $H\cup\cpa{\infty}$ with at most one singular point embedded in $S^3=\R^3\cup\cpa{\infty}$. If neither side of the $2$-sphere is a $3$-ball, then their union cannot be $S^3$.  This can be deduced from Eaton's Mismatch Theorem \cite{eaton} and a result of Bing on tame surfaces \cite{bing}.

\item In the piecewise linear category, Myers \cite{myers} showed that the CSI of $\R^3$ with itself yields uncountably many distinct $3$-manifolds.

\item While $\R^3 - \Int \nu r \not\approx \R^3_+$ in Example~\ref{ex:example}, its interior \emph{is} diffeomorphic to $\R^3$.  In particular, $M$ is contractible.  More generally, if $L \subset \R^3$ is a smooth proper multiray with at most countably many components, then $\R^3 - L \approx \R^3$. To see this, it suffices to prove that each compact $K \subset \R^3 - L$ is contained in a ball.  So, let $K \subset \R^3 - L$ be compact.  Let $B \subset \R^3$ be a ball containing $K$.  Let $F \colon \R^3 \times [0,1] \to \R^3$ be an ambient isotopy such that: (i) $F_0 = \tn{id}$, (ii) $F_t(L) \subset L$ for each $t\in [0,1]$, (iii) $\left.F_t\right|K = \tn{id}$ for each $t\in [0,1]$, and (iv) $F_1(L)$ is disjoint from $B$.  Such an $F$ is obtained by integrating a suitable vector field tangent to $L$ and vanishing on $K$.  The required ball is $F_1^{-1}(B)$.
\end{enumerate}
\end{remarks}

\section*{Acknowledgement}
The first author thanks Larry Siebenmann for introducing him to ladder manifolds of dimension $m\geq 7$, with stringers based on products of spheres, during their earlier collaboration~\cite{cks}.

\end{document}